\documentclass[11pt]{amsart}
\usepackage{graphicx} % Required for inserting images
\usepackage[latin1]{inputenc}
\usepackage{amsmath, amsfonts, amssymb, amsthm, mathtools, mathrsfs, bm}
\usepackage{enumerate}
\usepackage[shortlabels]{enumitem}
\usepackage[numbers]{natbib}
\usepackage{setspace}
\usepackage{xcolor}
\usepackage{tikz-cd}
\usepackage{verbatim}

\usepackage[pagebackref,colorlinks,linkcolor=blue,citecolor=blue,urlcolor=blue,hypertexnames=true]{hyperref}

\newtheorem{theorem}{Theorem}[section]
\newtheorem{lemma}[theorem]{Lemma}
\newtheorem{proposition}[theorem]{Proposition}
\newtheorem{corollary}[theorem]{Corollary}

\theoremstyle{definition}
\newtheorem{definition}[theorem]{Definition}

\newtheorem{remark}[theorem]{Remark}
\newtheorem{notation}[theorem]{Notation}

\newcommand{\Gl}{{\rm GL}}
\newcommand{\GL}{{\rm GL}}
\newcommand{\tDelta}{\widetilde{\Delta}}
\newcommand{\G}{\Gamma}

\newcommand{\ep}{\varepsilon}

\newcommand{\N}{\mathbb{N}}
\newcommand{\Z}{\mathbb{Z}}
\newcommand{\Q}{\mathbb{Q}}
\newcommand{\R}{\mathbb{R}}
\newcommand{\C}{\mathbb{C}}
\newcommand{\HH}{\mathbb{H}}
\DeclareMathOperator{\im}{Im}

\newcommand{\Tr}{\mathrm{Tr}}

\DeclareMathOperator{\Hom}{Hom}

\definecolor{darkpastelred}{rgb}{0.76, 0.23, 0.13}
\definecolor{darkred}{rgb}{0.55, 0.0, 0.0}
\definecolor{darkmagenta}{rgb}{0.55, 0.0, 0.55}
\definecolor{coolblack}{rgb}{0.0, 0.18, 0.39}
\definecolor{ceruleanblue}{rgb}{0.16, 0.32, 0.75}

\title[Central extensions and almost representations]{Central extensions and almost representations
}
\author{Marius Dadarlat and Forrest Glebe}\address{MD: Department of Mathematics, Purdue University, West Lafayette, IN 47907, USA} %\email{mdd@purdue.edu}
\address{FG: Department of Mathematical Sciences,
University of Hawaii at Manoa,
Honolulu, HI 96822, USA }

\date{\today}
\thanks{M.D. was partially supported by NSF grant \#DMS-2247334
\\
F.G. was supported by a Ross-Lynn grant in the 2023-2024 academic year.
}
\begin{document}

	%\tableofcontents
\begin{abstract}
{For a sequence of unital tracial $C^*$-algebras $(A_n,\tau_n),$ we  construct a canonical central extension of the unitary group $U(\ell^\infty (\mathbb{N},A_n)/c_0(\mathbb{N},A_n))$ 
by $Q(\mathbb{R})=c_0(\mathbb{N},\mathbb{R})/\mathbb{R}^\infty,$ using de la Harpe-Skandalis pre-determinant. For an asymptotic group homomorphism $\rho_n : \Gamma \to U(A_n),$ the corresponding pullback of the canonical central extension gives a 2-cohomology class in $H^2(\Gamma,Q(\mathbb{R}))$ which obstructs the perturbation of $(\rho_n)$ to a sequence of true homomorphisms of groups $\pi_n:\Gamma \to GL(A_n)$. The pairing of the obstruction class with elements of $H_2(\Gamma,\mathbb{Z})$ yields numerical invariants in $\tau_{n\,*} (K_0(A_n))$ that subsume the 
winding number invariants of Kazhdan, Exel and Loring.
 For generality, we allow bounded asymptotic homomorphisms to map the group $\Gamma$ into the general linear group of any sequence of tracial unital Banach algebras. In that case, the obstruction class belongs to $H^2(\Gamma,Q(\mathbb{C})),$ where $Q(\mathbb{C})=c_0(\mathbb{N},\mathbb{C})/\mathbb{C}^\infty.$ As an application, we show that 2-cohomology obstructs various stability properties under weaker assumptions than those found in existing literature.
 In particular, we show that the full group $C^*$-algebra  $C^*(\Gamma)$ of a discrete group $\Gamma$ is not $C^*$-stable if $H^2(\Gamma,\mathbb{R})\neq 0$ and in fact, $\G$ is not stable in operator norm with respect to tracial von Neumann algebras.
}
\end{abstract}
 \maketitle
	\section{Introduction}
 There are several known methods to construct  almost representations of a discrete group \(\Gamma\) that are far from true representations in the operator norm. At their core, these constructions are tied to certain cohomological invariants of \(\G\). Voiculescu  and Kazhdan  devised almost representations, relying either implicitly \cite{Voi:unitaries} or explicitly \cite{Kazhdan-epsilon}  on nontrivial central extensions of \(\Gamma\).
Ioana, Spaas and Wiersma \cite{Ioana-S-W:coho}
	used projective representations and 2-cohomology  to prove 
 the failure of lifting properties for full $C ^*$-algebras of countable groups with (relative) property $( \mathrm{T} )$.
 Burger, Ozawa, and Thom in \cite{BOT} built uniform almost representations that cannot be uniformly perturbed to true representations, using nonzero elements in the kernel of the comparison map \(H^2_b(\Gamma, \mathbb{R}) \rightarrow H^2(\Gamma, \mathbb{R})\) between bounded and standard group cohomology.

 As noted in the paper \cite{DECHIFFRE}  of de Chiffre, Glebsky,  Lubotzky, and Thom, it is known that the question of approximating asymptotic representations by true representations is related to a question about splitting group extensions. In both papers \cite{DECHIFFRE} and \cite{Ulam-new}, the authors use controlled asymptotics to construct extensions with abelian kernels. Then, they use the vanishing of 2-cohomology with abelian (nontrivial \(\Gamma\)-module) coefficients to improve the defects of asymptotic representations, eventually proving stability.

In this paper we abelianize the kernel of the extension associated with a sequence of asymptotic homomorphisms using the de la Harpe-Skandalis determinant. Due to the tracial property, we obtain in fact a central extension. We then use nonvanishing of 2-cohomology to show non-stability
 by this method.

Specifically, for a 2-cocycle $\sigma\in Z^2(\G,\R)$,  set $\omega_n=e^{2\pi i \sigma/n}\in Z^2(\G,\mathbb{T})$, and consider  the canonical sequence of unital maps
$\rho_n:\G \to U(L(\G,\omega_n))$ to the unitary groups of twisted group von Neumann algebras $L(\G,\omega_n)$. 
These maps factor through both the full and the reduced twisted $C^*$-algebras 
$C^*(\G,\omega_n)$ and $C^*_r(\G,\omega_n)$ and they constitute 
an asymptotic homomorphism. Indeed, since
$\rho_n(s)\rho_n (t)\rho_n (st)^{-1}=e^{2 \pi i \sigma(s,t)/n} 1_n$ we have that
\[\lim_{n\to \infty}\|\rho_{n}(s)\rho_{n}(t)-\rho_{n}(st)\|= 0, \quad\forall \, s,t \in \G.\] 
Moreover, if the 2-cocycle $\sigma$ is a bounded function, then
\[\lim_{n\to \infty} \sup_{s,t\in \G}\|\rho_{n}(s)\rho_{n}(t)-\rho_{n}(st)\|= 0.\]
\begin{theorem}\label{thm-vN}
  Let  $\G$ be a discrete countable group. Let $\sigma$ be a normalized 2-cocycle with $[\sigma]\in H^2(\G,\R)\setminus \{0\}.$ For the canonical sequence of maps $\rho_n:\G \to U(L(\G,e^{2\pi i \sigma/n}))$,  there exists no sequence of group homomorphisms $\pi_n:\G \to GL(L(\G,e^{2\pi i \sigma/n}))$ such that 
  $\lim_{n\to \infty}\left\|\rho_{n}(s)- \pi_{n}(s)\right\|=0,$ for all $s\in \G$.
\end{theorem}
As a corollary, we obtain the following
 \begin{theorem}\label{thm:Ai} Let  $\G$ be a discrete countable group.
	\begin{itemize}
		\item[(1)] If $ H^2(\Gamma,\R)\neq 0$, then $\Gamma$ is not local-to-local stable with respect the class of separable unital tracial $C^*$-algebras.
		\item[(2)] If the comparison map $J: H^2_b(\Gamma,\R)\to H^2(\Gamma,\R)$ is nonzero, then $\Gamma$ is not uniform-to-local stable with respect the class of separable unital tracial $C^*$-algebras.
	\end{itemize} 
\end{theorem}
  We refer the reader to Definition~\ref{def:Cstar-stable} for these notions of stability.
  The comparison map $J: H^2_b(\Gamma,\R)\to H^2(\Gamma,\R)$ is known to be surjective for all  hyperbolic groups, \cite{Mineyev}.
Thus from Theorem~\ref{thm:Ai} we derive the following:
 \begin{corollary}
     \label{thm:AA}
	Let $\G$ be a  hyperbolic group.
	If $H^2(\G,\R)\neq 0$,
	then  $\Gamma$ is not uniform-to-local stable with respect the class of unital tracial $C^*$-algebras.
 \end{corollary}
Assuming furthermore that the twisted group $C^*$-algebras of $\G$ admit MF quotients, we derive nonstability properties of $\G$ with respect to unitary groups $U(n)$, see Corollary~\ref{cor:321} and
\begin{theorem}\label{tprop:Ai}
	If $\G$ is a countable discrete group and for some $[\sigma]\in H^2(\Gamma,\R)\setminus \{0\}$  there  is a sequence 
	$\theta_n \searrow 0$  so that for each $n$,  the twisted full group $C^*$-algebra, $C^*(\Gamma,\sigma_{\theta_n})$ has a nonzero MF quotient, then $\Gamma$ is not matricially stable.
\end{theorem}

 Our approach is based on the construction of a canonical central extension
 \begin{equation}\label{eq:0001}
	\begin{tikzcd}
		1 \ar[r]&	Q(\C)  \ar[r]& E \ar[r]&  \GL(A)\ar[r] \ar[r] & 1
	\end{tikzcd}
\end{equation}
where $Q(\C)=c_0(\N,\C)/c_{00}(\N,\C)$ and where $GL(A)$ is the  group  of invertible elements of the corona $C^*$-algebra $A=\ell^\infty(\N,A_n)/c_{0}(\N,A_n)$ for a given sequence of unital tracial $C^*$-algebras $(A_n,\tau_n)$ such that $\sup_n \|\tau_n\|<\infty$. For a group homomorphism $\rho:\G \to \GL(A),$ the pullback of the extension \eqref{eq:0001} is a central extension whose cohomology class, denoted $[\rho]\in H^2(\G,Q(\C)),$
obstructs the lifting of $\rho$ to a sequence a representations $\pi_n:\G \to \GL(A_n)$.
The construction of \eqref{eq:0001} is based on the de la Harpe-Skandalis (pre)determinant. For added generality, we allow $A_n$ to be unital tracial Banach algebras. If $A_n$ are $C^*$-algebras and one replaces $\GL(A)$ by the unitary group $U(A),$ then one obtains a central extension as above with $Q(\R)=c_0(\N,\R)/c_{00}(\N,\R)$ in place of $Q(\C)$. Furthermore we show that
the class $[\rho]\in H^2(\G,Q(\R))$ can be lifted to a class in $H^2(\G,H)$ where 
\[H:=\frac{c_0(\N,\tau_{n*}\left(K_0(A_n)\right)}{c_{00}(\N,\tau_{n*}\left(K_0(A_n)\right)}.\]
This means that the components $\langle \rho_n,c\rangle$ of the Kronecker paring between $[\rho]$ and any 2-homology  class $[c]\in H_2(\G,\Z)$ are eventually elements of $\tau_{n*}(K_0(A_n))$.
We connect these components to push-forward images of elements of $K_0(\ell^1(\G))$ under $\rho_n$  by invoking an index theorem from \cite{DDD}, see {Corollary}~\ref{index}.

While it was already known that the non-vanishing of $H^2(\G,\R)$ implies matricial nonstability for large classes of groups satisfying certain geometric conditions, see for example \cite{ESS-published}, \cite{CCC}, \cite{DDD},\cite{arXiv:2204.10354},\cite{MR4600056}, \cite{EEE}, the viewpoint that we promote here leads to nonstability properties for general discrete groups with nonvanishing 2-cohomology.

The paper is organized as follows.
 Section~\ref{sec:2}  reviews the de la Harpe-Skandalis determinant and 2-(co)homology. Section~\ref{sec:3}  is devoted to the central extension \eqref{eq:0001} and the cohomology classes of its pullbacks arising from asymptotic homomorphisms.  
Section~\ref{sec:4}  discusses the connection with K-theory invariants associated an asymptotic homomorphisms.
In Sections~\ref{sec:5}  and ~\ref{sec:6}  we derive non-stability properties arising from nonvanishing of $H^2(\G,\R)$.
\section{Preliminaries}\label{sec:2}
Let $A$ be a complex Banach algebra endowed with a continuous tracial linear map $\tau:A \to F$ to a complex Banach space $F.$ Thus $\tau(xy)=\tau(yx)$ for all $x,y\in A.$
The canonical extension of $\tau$ to the algebra $M_\infty(A)$ is denoted again by $\tau$. If $F=\C$, then the map $\tau$ is called a trace and we say that $A$ is a tracial Banach algebra.

If \( A \) has a unit, \( GL_n(A) \) represents the group of invertible elements in \( M_n(A) \), equipped with the topology induced by the norm. If \( A \) lacks a unit, let \( \widetilde{A} \) denote the algebra obtained by adjoining a unit, which becomes a Banach algebra under a suitable norm. In this case, \( GL_n(A) \) refers to the topological subgroup of \( GL_n(\widetilde{A}) \) consisting of elements of the form \( 1 + a \), where \( a \in M_n(A) \). In all instances, \( GL_\infty(A) \) denotes the inductive limit of \( (GL_n(A))_{n \geq 1} \), with respect to the inclusions \( a \longmapsto \left(\begin{array}{ll}a & 0 \\ 0 & 1\end{array}\right) \). This is a topological group, and its connected component is denoted by \( GL_\infty^0(A) \).

\begin{definition}\label{def:LLtau} Define 
$L_\tau:\{u\in GL_\infty(A):\|u-1\|<1 \}\to F$ by 
\[L_\tau(u)=\frac{1}{2\pi i}\tau(\log(u))\]
 where $\log$ is defined using a power series centered at 1, $\log(u)=-\sum_{n=1}^\infty (1-u)^n/n$.
\end{definition}
Note that $\|\log(u)\|\leq 2\|u-1\|$ if $\|u-1\|<1/2.$

In order to prove certain properties of $L_\tau,$ it will be helpful to appeal to de la Harpe-Skandalis pre-determinant.
\begin{definition}[\cite{paper:DeLaHarpeSkandalis}]
Let $\xi:[0,1]\rightarrow \Gl_\infty^0(A)$ be a piecewise-smooth, continuous function. The de la Harpe-Skandalis pre-determinant is defined by
$$\tDelta_\tau(\xi)=\frac{1}{2\pi i}\tau\left(\int_0^1\dot{\xi}(t)\xi(t)^{-1}dt\right)\in F.$$
\end{definition}

\begin{proposition}[\cite{paper:DeLaHarpeSkandalis}]\label{dlhsprop} For paths  in $\Gl_\infty^0(A),$ we have
\begin{enumerate}[(a)]   
\item If multiplication of paths is defined pointwise, $\tDelta_\tau(\xi\cdot\eta)=\tDelta_\tau(\xi)+\tDelta_\tau(\eta)$.
    \item If $||\xi(t)-1||<1$ for all $t,$ then $\tDelta_\tau(\xi)=\frac{1}{2\pi i}(\tau(\log(\xi(1))-\log(\xi(0)))$.
    \item If $\xi$ is homotopic to $\eta$ with fixed endpoints, then $\tDelta_\tau(\xi)=\tDelta_\tau(\eta)$. 
    \item  If $v\in\Gl_\infty(A)$, then $\tDelta_\tau(v\xi v^{-1})=\tDelta_\tau(\xi)$.
    \item Let $p \in M _{\infty}(A)$ be an idempotent and let $\xi_p$ be the loop $\xi_p(t)=(1-p)+e^{2\pi i t}p$, $t\in [0,1]$; then $\tDelta(\xi_p)=\tau(p).$
\end{enumerate}
\end{proposition}
For $u\in \GL_\infty(A)$ with $\|u-1\|<1,$ set $\xi_u(t)=(1-t)1+tu$. Then by {\em (b)} above, it follows that 
\begin{equation}\label{def:Ltau}
    L_\tau(u)=\tDelta_\tau(\xi_u).
\end{equation}

\begin{proposition}\label{prop-L} Let $A$ and $\tau:A \to F$ be as above. Then
\begin{enumerate}[(1)]
    \item If $||u_i-1||<1/4,$ $i=1,2,$ then $L_\tau(u_1u_2)=L_\tau(u_1)+ L_\tau(u_2)$.
    \item For any  $u,v\in\Gl_\infty(A)$ with $\|u-1\|<1,$ we have $L_\tau(u)=L_\tau(vuv^{-1}).$ 
    \item If $\|u-1\|<1,$ then  $L_\tau(u^{-1})=-L_\tau(u)$.
\end{enumerate}
\end{proposition}

\begin{proof}
{(1)} By the above discussion, it suffices to show that 
$$\tDelta_\tau(\xi_{u_1u_2})=\tDelta_\tau(\xi_{u_1})+\tDelta_\tau(\xi_{u_2}).$$
By Proposition~\ref{dlhsprop} {\em(a)}, the right side is equal to
$\tDelta_\tau(\xi_{u_1}\xi_{u_2}).$
By Proposition~\ref{dlhsprop} {\em(c)} it suffices to find a fixed-endpoint homotopy from $\xi_{u_1u_2}$ to $\xi_{u_1}\xi_{u_2}$. To do this, define
$$H(s,t)=s\xi_{u_1 u_2}(t)+(1-s)\xi_{u_1}(t)\xi_{u_2}(t).$$
One checks immediately that  $\|H(s,t)-1\|<1,$  so we are done.

{(2)} Note that
{\[L_\tau(vuv^{-1})=\tDelta_\tau(\xi_{vuv^{-1}})=\tDelta_\tau(v \xi_uv^{-1})=\tDelta_\tau(\xi_u)=L_\tau(u)\]}

{(3)} If $\|u-1\|<1,$ then
{\[L_\tau(u^{-1})=\tDelta_\tau(\xi_{u^{-1}})=\tDelta_\tau(u^{-1}\cdot \xi_{u}(1- t))=\tDelta_\tau(u^{-1})+\tDelta_\tau(\xi_{u}(1- t))=0-\tDelta_\tau(\xi_u)=-L_\tau(u)\]}
\end{proof}

\begin{definition}[\cite{paper:DeLaHarpeSkandalis}]
The de la Harpe-Skandalis determinant is the group homomorphism defined by the mapping 
$$
\Delta_\tau: \GL_{\infty}^0(A) \longrightarrow F / {\tau_*}\left(K_0(A)\right)
$$
which associates to an element $u$ in the domain, the class modulo $\tau_*\left(K_0(A)\right)$ of $\widetilde{\Delta}_\tau(\xi)$, where $\xi$ is any piecewise differentiable path in $\GL _{\infty}^0(A)$ from $1$ to $u$. If $A$ is not unital, one extends $\tau$ to $\tilde{A}$ by putting $\tau(1)=0$.
\end{definition}
Since the range of $\Delta_\tau$ is an abelian group, $\Delta_\tau$ vanishes on the commutator subgroup \\ $[\GL_{\infty}^0(A),GL_{\infty}^0(A)]=[GL_{\infty}(A),GL_{\infty}(A)]$. The equality of these commutator subgroups is a well-known consequence of  Whitehead's lemma.

We will use homology and cohomology with coefficients in abelian groups $Q$ viewed as trivial $\Gamma$-modules. The reader is referred to \cite[Chapter II.3]{iBrown:book-cohomology} for more background information.
Let $C_k(\Gamma)$ consist of formal linear combinations of elements of~$\Gamma^k$ with coefficients in $Q$. We write a typical element of $C_2(\Gamma)$ as
$$\sum_{j=1}^m k_j[a_j|b_j]$$
with $a_j,b_j\in\Gamma$ and $k_j\in Q$. There are boundary maps $\partial_2:C_2(\Gamma)\rightarrow C_1(\Gamma)$ defined by
$$\partial_2[a|b]=[a]-[ab]+[b]$$
and $\partial_3:C_3(\Gamma)\rightarrow C_2(\Gamma)$ defined by
$$\partial_3[a|b|c]=[b|c]-[ab|c]+[a|bc]-[a|b].$$
Then $H_2(\Gamma;Q):=\ker(\partial_2)/\im(\partial_3)$. An element of $Z_2(\G,Q):=\ker(\partial_2)$ is referred to as a {2-cycle} and an element in $\im(\partial_3)$ is referred to as a {2-boundary}.

Let us recall now the definition of $2$-cohomology $H^2(\G,Q).$
A 2-cocycle $\sigma:\G^2 \to Q$ 
is a function that satisfies the equation
\begin{equation}\label{eq:cocyle}
	\sigma(a,b)+\sigma(ab,c)=\sigma(a,bc)+\sigma(b,c),\quad \text{for all}\quad  a,b,c \in \G.
\end{equation}
The group of cocycles is denoted by $Z^2(\G,Q).$
A {2-coboundary} is a 2-cocycle  that can be written in the form
$$\sigma(a,b)=\partial\, \gamma (a,b)=\gamma(a)-\gamma(ab)+\gamma(b)$$ for some function $\gamma:\G \to Q$. $H^2(\Gamma;Q)$ is defined to be the group of 2-cocycles, mod the subgroup of 2-coboundaries. The group operation is pointwise addition.
By replacing $\sigma$ by $\sigma+\partial \,\gamma,$ where $\gamma:\G \to Q$ is defined by $\gamma(a)=-\sigma(e,a)$ for $a\in \G,$ one can obtain a 2-cocycle satisfying 
\[\sigma(a,e)=\sigma(e,a)=0.\]
A cocycle satisfying the above equation is called {\em normalized}.

The {Kronecker pairing} is the bilinear map $H^2(\Gamma;Q)\times H_2(\Gamma;\Z)\to Q$ defined by 
$$\langle [\sigma],[c]\rangle=\sum_{j=1}^m k_j\sigma(a_j,b_j),$$
where $c=\sum_{j=1}^m k_j[a_j|b_j]\in Z_2(\G,\Z)$.  By the universal coefficient theorem, if $Q$ is divisible, then the {Kronecker pairing} induces an isomorphism $H^2(\Gamma;Q)\cong\Hom(H_2(\G,\Z),Q)$. 

By a classic result in algebra \cite{iBrown:book-cohomology}, the second cohomology group \( H^2(\Gamma, Q) \) classifies all central extensions of \( \Gamma \) by \( Q \). 
Given a central extension,
\[0\to Q \stackrel{j}{\longrightarrow} E \stackrel{q}{\longrightarrow}\G \to 1\]
we can associate to it a normalized 2-cocycle $\sigma$  by choosing a unital set theoretic section $\gamma$ of $q$
and define $\sigma(a,b)=j^{-1}(\gamma(a)\gamma(b)\gamma(ab)^{-1}).$
Conversely, if a normalized 2-cocycle $\sigma$ is given, one constructs a central extension
where the set $E=Q \times \G$ is endowed with multiplication $(x,a)\cdot (y,b)=(x+y+\sigma(a,b),ab)$, and $j(x)=(x,1),$ $q(x,a)=a,$  $x,y\in Q$ and $a,b\in \G.$
We also consider the section $\gamma:\G \to E,$ $\gamma(a)=(0,a).$
Then $\gamma(a)\gamma(b)\gamma(ab)^{-1}=(\sigma(a,b),1)=j(\sigma(a,b)).$
\section{Canonical central extensions}\label{sec:3}

Let $(A_{n})_n$ be a sequence of unital  Banach algebras.
Consider the Banach algebra 
$$A=\ell^\infty(\N,A_n)=\{(a_n)_n\in \prod A_n:\, \sup\nolimits_{n}\|a_n\|<\infty\}$$ and its two-sided  ideals $$J=c_0(\mathbb{N},A_n)=\{(a_n)_n\in \prod A_n:\, \lim\nolimits_{n} \|a_n\|=0\}.$$
$$J_0=c_{00}(\mathbb{N},A_n)=\{(a_n)_n\in \prod A_n:\,\, \exists k \,\, \text{such that} \, \,a_n=0\,\,  \text{for}\,\, n\geq k\}.$$
$J_0\subset J$ and $J$ is closed in $A$.
 The unit of $A_n$ is denoted by $1_n$.
Consider the following groups (we will adjoin the unit of $A$ to $J$ in the definitions of $\GL(J)$ and $\GL(J_0)$):
$$
\begin{aligned}
&GL(A)\cong P =\{\left(u_{n}\right)_{n}
\in\prod_{n=1}^{\infty} \GL(A_{n}): \,\,\sup\nolimits_{n}||u_n||<\infty \,\, \text{and} \,\, \sup\nolimits_{n}||u_n^{-1}||<\infty\}, \\ 
&GL(J)\cong P_{1}=\left\{\left(u_{n}\right)_{n}\in P:\left\|u_{n}-1_n\right\| \rightarrow 0\right\}, \\
&GL(J_0)\cong P_{0}=\bigoplus_{n=1}^{\infty} G L\left(A_{n})  =\left\{\left(u_{n}\right)_{n}\in P_1:  \,\, \exists k \,\, \text{such that}\,\, u_{n}=1_{n} \text { for } n \geqslant k\right\}\right.
\end{aligned}
$$

The inclusions of normal subgroups $P_{0} \subset P_{1} \subset P$ give an exact sequence of groups
\begin{equation}\label{eq:0000}
 1 \rightarrow \GL(J)/\GL(J_0) \rightarrow \GL(A)/\GL(J_0) \stackrel{\nu}\longrightarrow \GL(A/J) \rightarrow 1
\end{equation}
\begin{equation}\label{eq:000} \text{or equivalently}\qquad
 1 \rightarrow P_1/P_0 \rightarrow P/P_0 \stackrel{\nu}\longrightarrow P/P_1 \rightarrow 1
\end{equation}
We will abuse notation and write elements of $P/P_0$ as $(u_n)_n$ with the understanding that finitely many coordinates $u_n$ are neither specified nor determined. 

{
Suppose now that each Banach algebra $A_n$  admits a  trace $\tau_n$ such that $\sup_n \|\tau_n\|<\infty.$
Then we define a continuous tracial linear map
\[\tau:A \to \ell^\infty(\N,\C),\quad (a_n)_n\mapsto (\tau_n(a_n))_n.\]
}
Note that $\tau(J)\subset c_0(\N,\C).$
We use the map $L_{\tau}$  from Definition~\ref{def:LLtau}
to construct a push-out of the extension \eqref{eq:000} to a central extension.
 The group $P_1/P_0$ has a $P/P_0$  action given by
 $w\cdot u=wuw^{-1}$ with $w\in P/P_0$ and $u\in P_1/P_0$.
We view  the group 
$$Q(\C):=c_0(\N,\C)/{c_{00}(\N,\C)}$$ 
 as a trivial $P/P_0$-module.
\begin{lemma} The map  $L_{\tau}:\{u\in \GL^0(A):\|u-1\|<1\}\to \ell^\infty(\N,\C)$ from Definition~\ref{def:LLtau} induces
	 a $P/P_0$-equivariant group homomorphism $L: {P_{1}}/{P_{0}} \rightarrow Q(\C)$ such that 
 \begin{equation}\label{eq:defL}  L\left(\left(u_{n}\right)_{n}\right)=\left(\frac{1}{2 \pi i}  \tau_{n}\left(\log u_{n}\right)\right)_{n}.
 \end{equation} 
  
\end{lemma}
\begin{proof} Note that for any $\ep>0,$ each coset   $[u]\in \GL(J)/\GL(J_0)$ contains elements $u\in \GL^0(J)$ with $\|u-1\|<\ep.$ Define $L[u]:=L_\tau(u)+c_{00}(\N,\C)\in Q(\C).$
$L_\tau(u)\in c_0(\N,\C)$ since 
$|\tau_n\left(\log u_{n}\right)|\leq 2\|\tau_n\|\|u_n-1\|$ 
if $\|u_n-1\|<1/2$.
Proposition~\ref{prop-L}, (1)-(2) and the functorial properties of $\tDelta_\tau$ imply the desired properties. 
\end{proof}

Using $L,$ we consider the push-out of the extension \eqref{eq:000} described by the diagram: 
\[
\begin{tikzcd}
	1 \ar[r]&	P_1/P_0\ar[r] \ar[d,"L"]& P/P_0\ar["\bar{L}", d] \ar[r]& P/P_1 \ar[r] & 1\\
	1 \ar[r]&	Q(\C )\ar[r]& E \ar["\nu", r]&  P/P_1\ar[r] \ar[equal]{u}\ar[r] & 1
\end{tikzcd}
\]
Here $E\cong (Q(\C)\times P/P_0)/\{(-L(u),u):u\in P_1/P_0\},$ $\bar{L}(w)=(0,w)$ and $\nu(r,w)=w$.

Since $Q(\C)$ is a trivial $P/P_0$-module,  it is also a trivial $E$-module, so that the push-out is central extension.
\begin{definition}
For a sequence of  tracial Banach algebras $(A_n,\tau_n)$ with $\sup_n\|\tau_n\|<\infty$,
\begin{equation}\label{eq:003}
	\begin{tikzcd}
		1 \ar[r]&	Q(\C)  \ar[r]& E \ar[r]&  P/P_1\ar[r] \ar[r] & 1
	\end{tikzcd}
\end{equation} where $P/P_1\cong GL(\ell^\infty(\N,A_n)/c_0(\N,A_n)),$ 
is the central extension canonically associated to $(A_n,\tau_n)$.   The corresponding 2-cohomology class is denoted by $[E]\in H^2(P/P_1,Q(\C))$.
\end{definition}
\begin{remark}\label{eq:asymptotic} Let $\Gamma$ be a discrete group. By general Banach algebra theory one verifies that the following conditions are equivalent
for a sequence of unital maps $(\rho_n)_n$: 
\[\rho_{n}: \Gamma \rightarrow G L\left(A_{n}\right).\]
    \begin{itemize}
        \item[(1)]  $\lim _{n \rightarrow \infty}\left\|\rho_{n}(a) \rho_{n}(b)-\rho_n(ab)\right\|=0 \, \text{and} \, \sup_n||\rho_n(a)||<\infty,\, a,b \in \G.$
        
         \item[(2)] $ \lim _{n \rightarrow \infty}\left\|\rho_{n}(a) \rho_{n}(b) \rho_{n}(a b)^{-1}-1_{n}\right\|=0 \, \text{and} \, \sup_n||\rho_n(a)||<\infty,\, a,b \in \G.$

         \item[(3)] $\lim _{n \rightarrow \infty}\left\|\rho_{n}(a) \rho_{n}(b)-\rho_n(ab)\right\|=0 \,\text{and} \, \sup_n(||\rho_n(a)||+||\rho_n(a)^{-1}||)<\infty,\, a,b \in \G.$ 
    \end{itemize}
    \end{remark}
Thus,
we can view the sequence $(\rho_n)_n$  as a unital map $\rho: \Gamma \rightarrow P$ with the property that the composition 
$$
\Gamma \xrightarrow{\rho} P \xrightarrow{\nu} P / P_{1}
$$
defines  a group homomorphism $\dot{\rho}=\nu \circ \rho:\G \to P/P_1$.
We construct the pull-back of the extension  \eqref{eq:003} as described in diagram below,
\[\begin{tikzcd}
	1 & Q(\C) && E  && P/P_1  & 1 \\
	1 & Q(\C) \ar[equal]{u} && E_\rho && \Gamma & 1
	\arrow[from=1-1, to=1-2]
	\arrow[from=1-2, to=1-4]
	\arrow["\nu", from=1-4, to=1-6]
	\arrow[from=1-6, to=1-7]
	\arrow[from=2-1, to=2-2]
	%\arrow[from=2-2, to=1-2]
	\arrow[from=2-2, to=2-4]
	\arrow[ from=2-4, to=1-4]
	\arrow["\nu'",from=2-4, to=2-6]
	\arrow["\dot{\rho}",swap, from=2-6, to=1-6]
	\arrow[from=2-6, to=2-7]
\end{tikzcd}\]
 where
$
E_{\rho}=\left\{(e, a) \in E \times \Gamma: \nu(e)=\dot{\rho}(a)\right\}
$
and  $\nu^{\prime}(e, a)=a$.
The isomorphism class of the central extension
\begin{equation}\label{eq:002}
	1 \rightarrow Q(\C) \rightarrow E_{\rho} \rightarrow \Gamma \rightarrow 1
\end{equation}
depends solely on $\dot{\rho}$. {From this, it follows that if $(\rho_n)_n$ and $(\varphi_n)_n$ are two  sequences of unital maps as in Remark~\ref{eq:asymptotic} with the additional condition that $||\rho_n(a)\varphi_n(a)^{-1}-1||\rightarrow0$, 
then $\dot{\rho}=\dot{\varphi}$ and hence  the extension associated to each sequence is the same}. In particular, 
observe that if there is a sequence of representations $\pi_{n}: \Gamma \rightarrow G L\left(A_{n}\right)$ such that
\begin{equation}\label{eq:001}
  \left\|\rho_{n}(a) \pi_{n}(a)^{-1}-1\right\| \rightarrow 0, \quad \text{for all}\,\, a \in \G,
\end{equation}
this will induce a group homomorphism
$\pi: \Gamma \rightarrow P / P_{0}\stackrel{\bar{L}}{\longrightarrow} E,$ which in turn will produce a splitting $s:\Gamma \to E_\rho$ of the  extension \eqref{eq:002},
given by $s(a)=(\pi(a), a)$, $a\in \G$. Note that $s(a)\in E_{\rho}$ since
$\rho(a) \pi(a)^{-1} \in P_{1}/P_0$ by condition \eqref{eq:001}.

It follows by construction, that the central extension \eqref{eq:002} is associated to the 2-cocycle
\begin{equation}\label{eq:2-cocycle}
\omega(a,b)=(\omega_n(a,b))_n=\left(\frac{1}{2 \pi i}  \tau_{n}\left(\log \left(\rho_n(a)\rho_n(b)\rho_n(ab)^{-1}\right)\right)\right)=L(\rho(a)\rho(b)\rho(ab)^{-1}).
\end{equation}

\begin{notation}
    The 2-cohomology class of the extension \eqref{eq:002} is the class of 2-cocycle $\omega$  and is denoted by $[{\rho}]\in H^2(\G,Q(\C))$. Evidently, we have $[{\rho}]=\dot{\rho}^*[E].$
\end{notation}

By Remark~\ref{eq:asymptotic},  condition \eqref{eq:001} is equivalent to 
$\,\, \left\|\rho_{n}(a)- \pi_{n}(a)\right\| \rightarrow 0, \,\, a \in \G.$
The above discussion proves the following theorem:
\begin{theorem}\label{thm:basic}
Let $(A_n,\tau_n)$ be a sequence of  tracial Banach algebras  with $\sup_n\|\tau_n\|<\infty.$ Any sequence of unital maps $\rho_n:\G \to \GL(A_n)$
  satisfying $||\rho_{n}(a) \rho_{n}(b) -\rho_{n}(a b)||\rightarrow 0$ and 
  $\sup_n||\rho_n(a)||<\infty$ for all $a,b\in\G$,  defines a 2-cocycle $\omega \in Z^2(\G,Q(\C))$ given by equation ~\eqref{eq:2-cocycle}.
If there exists a sequence of group homomorphisms $\pi_n:\G\rightarrow\Gl(A_n)$ so that $||\rho_{n}(a)- \pi_{n}(a)||\rightarrow 0,$ then $\omega$ is cohomologous to 0, i.e. $[\rho]=0$ in $H^2(\G,Q(\C))$.
\end{theorem}

One can investigate the nontriviality of $[\rho]$ by pairing it with 2-cycles. Note that if 
\begin{equation}\label{eq:pre}
c=\sum_{j=1}^{m}k_j [a_j|b_j]\in Z_2(\G,\Z),\end{equation}
then
\begin{equation}\label{eq:pairr}
    \langle[\rho],[c]\rangle=L\left(\prod_{j=1}^m\left(\rho(a_j)\rho(b_j)\rho(a_jb_j)^{-1}\right)^{k_j}\right).
\end{equation}
\begin{corollary}\label{cor:proj-appl}
    Let $(A_n,\tau_n)$ and $(\rho_n)$ be as in Theorem~\ref{thm:basic}.
    Assume furthermore that there is a sequence of 2-cocycles $\sigma_n\in Z^2(\G,\mathbb{R})$ such that 
    $\rho_{n}(a) \rho_{n}(b) \rho_{n}(a b)^{-1}=e^{2\pi i\sigma_n(a,b)}1_{n}$ and
    $\sigma_n(a,b)\to 0$ for all $a,b\in \G$. Then $[\rho]$ is represented by the 2-cocycle $\omega\in Z^2(\G,Q(\C))$,  
    $\omega(a,b)=\left(\sigma_n(a,b)\tau_n(1_n)\right)_n.$ In particular if
    there is $c \in Z_2(\G,\Z)$ such that for infinitely many $n$, $\langle \sigma_n,c\rangle \neq 0$ and $\tau_n(1_n)\neq 0,$  then there exists no sequence of group homomorphisms $\pi_n:\G\rightarrow\Gl(A_n)$ so that $||\rho_{n}(a)- \pi_{n}(a)||\rightarrow 0,$ for all $a\in \G$.
\end{corollary}
\begin{proof}
    By Theorem~\ref{thm:basic},
	 the class $[\rho]\in H^2(\G,Q(\C))$
	is given by a 2-cocycle $\omega$ with components 
	\[\omega_n(a,b)=\frac{1}{2 \pi i}  \tau_{n}\left(\log \left(\rho_n(a)\rho_n(b)\rho_n(ab)^{-1}\right)\right)_{n}=\frac{1}{2 \pi i}  \tau_{n}\left(\log e^{2\pi i \sigma_n(a,b)}\right)=\sigma_n(a,b)\tau(1_n).\]
	The pairing 
	$H^2(\G,Q(\C))\times H_2(\G,\Z)\to Q(\C)$ yields
	\[\langle [\rho],[c]\rangle=(\langle\sigma_n,c\rangle\tau_n(1_n))_n\neq 0,\]
 and hence $[\rho]\neq 0$ in $H^2(\G,Q(\C))$. 
\end{proof}
The class $[\rho]$ satisfies the following invariance property. 
\begin{proposition}\label{prop:invariance}
   Suppose that $\phi_n: (A_n,\tau_n) \to (A'_n,\tau'_n)$ is a sequence of contractive unital homomorphisms of tracial Banach algebras as above, such that $\tau_n'\circ 
   \phi_n=\tau_n.$ 
   Define $\rho'_n:A_n\to A_n'$ by $\rho'_n=\phi_n\circ \rho_n.$
   Then $[\rho']=[\rho]\in H^2(\G,Q(\C)).$
   \end{proposition}
   In particular, if we establish that the sequence $(\rho_n)_n$ is not perturbable to a sequence of homomorphisms by verifying that $[\rho]\neq 0$,  then the same property holds for the sequence $(\rho'_n)_n$ even though $A_n'$ might be significantly larger then $A_n$, for example a finite von Neumann algebra completion of $A_n,$ or significantly smaller than $A_n$, for example a finite dimensional quotient of $A_n$.
\begin{proof}
    By functoriality of the  holomorphic calculus, since $\rho'_n=\phi_n\circ \rho_n,$ we have
    \[\tau'_{n}\left(\log \left(\rho'_n(a)\rho'_n(b)\rho'_n(ab)^{-1}\right)\right)=(\tau'_{n}\circ \phi_n)\left(\log \left(\rho_n(a)\rho_n(b)\rho_n(ab)^{-1}\right)\right).\]
    The desired conclusion follows now from \eqref{eq:2-cocycle}.
\end{proof}

\section{K-theory}\label{sec:4}
In  this section we connect the  2-cohomology class $[\rho]$ to K-theory invariants, see Proposition~\ref{prop:integral} and Corollary~\ref{index}.
These properties are not used for the non-stability results in the next two sections.

Let $\rho_{n}: \Gamma \rightarrow U\left(A_{n}\right)$ be
   an asymptotically multiplicative sequence of unital maps that satisfies the conditions from  Remark~\ref{eq:asymptotic}, where $A_n$ are tracial unital $C^*$-algebras. 

 Hopf's formula expresses the second homology of \(\Gamma\) as 
\[H_2(\Gamma,\mathbb{Z})=\frac{R\cap [F,F]}{[R,F]}\] 
in terms of a free presentation 
\(1 \to R \to F {\longrightarrow} \Gamma \to 1.\) 
 Each element \(x \in H_2(\Gamma,\mathbb{Z})\) can be represented by a product of commutators \(\prod_{i=1}^{g} [a_i,b_i]\) with \(a_i,b_i \in F\), for some integer \(g \geq 1\), such that \(\prod_{i=1}^{g} [\bar{a}_i,\bar{b}_i] = 1\), where $\bar a_i$ and $\bar b_i$ are the images in $\Gamma$ of ${a}_i$ and ${b}_i$.

There is an isomorphism $\varphi:\frac{R\cap [ F,F]}{[R,F]}\to H_2(\Gamma,\mathbb{Z})$  from the Hopf's realization of $H_2(\Gamma,\mathbb{Z})$ to the the same group defined via the bar-resolution. 
 By \cite[chapter II.5 Exercise 4]{iBrown:book-cohomology}, 
 if $r\in H_2(\G,\Z)$ is represented by $\prod_{i=1}^g[a_i,b_i]$ in the Hopf formula,
then a 2-cycle representative for the class of $\varphi(r)$ is the element $\sum_{i=1}^g d_i$,  where 
\begin{equation}\label{eq:hoppy}
    d_i=[I_{i-1}|\bar a_i]+[I_{i-1}\bar a_i|\bar b_i]-[I_{i-1}\bar a_i\bar b_i\bar a_i^{-1}|\bar a_i]-[I_i|\bar b_i]
\end{equation}
and $I_i=[\bar a_1,\bar b_1]\cdots[\bar a_i,\bar b_i]$. 

\begin{proposition}\label{prop:Hopf}
  Let  $r\in H_2(\G,\Z)$ be represented by $\prod_{i=1}^g[a_i,b_i]$ in the Hopf formula. Then
  \begin{equation}\label{eq:hoppyy}
      \langle[\rho],[\varphi(r)]\rangle=L\left(\prod_{i=1}^{g} [\rho(\bar{a}_i),\rho(\bar{b}_i)]\right)
  \end{equation}
\end{proposition}
\begin{proof}
    Let $[\varphi(r)]=\sum_{i=1}^g [d_i]$ with $d_i$ as in \eqref{eq:hoppy}. We compute
\begin{align*}
\langle[\rho],[d_i]\rangle=&L(\rho(I_{i-1})\rho(\bar a_i)\rho(I_{i-1}\bar a_i)^{-1})
+L(\rho(I_{i-1}\bar a_i)\rho(\bar b_i)\rho(I_{i-1}\bar a_i\bar b_i)^{-1})\\
&+L(\rho(I_{i-1}\bar a_i\bar b_i)\rho(\bar a_i)^{-1}\rho(I_{i-1}\bar a_i\bar b_i\bar a_i^{-1})^{-1})+L(\rho(I_i\bar b_i)\rho(\bar b_i)^{-1}\rho(I_i)^{-1})\\
=&L(\rho(I_{i-1})\rho(\bar a_i)\rho(\bar b_i)\rho(\bar a_i)^{-1}\rho(I_i{\bar b_i})^{-1}\rho(I_i\bar b_i)\rho(\bar b_i)^{-1}\rho(I_i)^{-1}) 
\\
=&L(\rho(I_{i-1})[\rho(\bar a_i),\rho(\bar b_i)]\rho(I_i)^{-1}).
\end{align*}
Consequently,
\[
\langle[\rho],[\varphi(r)]\rangle=\sum_{i=1}^g \,L(\rho(I_{i-1})[\rho(\bar a_i),\rho(\bar b_i)]\rho(I_i)^{-1})=L\left(\rho(I_1)\Big(\prod_{i=1}^g[\rho(\bar a_i),\rho(\bar b_i)]\Big)\rho(I_g)^{-1}\right)
.\]
Since $I_g=I_1=1$,  we obtain the desired conclusion.
\end{proof}
Consider the groups
\[\ell^\infty(\N,\tau_{n*}\left(K_0(A_n)\right)=\{(x_n)_n\in \ell^\infty(\N,\C): x_n\in \tau_{n*}\left(K_0(A_n)\right),\,\, \forall n\in \N\}. \]
\[c_0(\N,\tau_{n*}\left(K_0(A_n)\right)=\{(x_n)_n\in c_0(\N,\C): x_n\in \tau_{n*}\left(K_0(A_n)\right),\,\, \forall n\in \N\}. \]
\[c_{00}(\N,\tau_{n*}\left(K_0(A_n)\right)=\{(x_n)_n\in c_{00}(\N,\C): x_n\in \tau_{n*}\left(K_0(A_n)\right),\,\, \forall n\in \N\}. \]

Consider the following subgroup $H$ of $Q(\C)=c_0(\N,\C)/c_{00}(\N,\C)$  and the corresponding quotient group
 \[\quad H:=\frac{c_0(\N,\tau_{n*}\left(K_0(A_n)\right)}{c_{00}(\N,\tau_{n*}\left(K_0(A_n)\right)},\quad
  Q(\C)/H \subset \frac{\ell^\infty(\N,\C/\tau_{n*}\left(K_0(A_n)\right)}{c_{00}(\N,\C/\tau_{n*}\left(K_0(A_n)\right)}. \]
  Let $q:Q(\C)\to Q(\C)/H$ be the quotient map. Equation~\eqref{eq:defL} shows that the components of the map $\Delta:=q\circ L:P_1/P_0 \to Q(\C)/H$ can be expressed as:
  \[\Delta(\mathbf{u})=\left(\Delta_{\tau_n}(u_n)\right)_n\]
\begin{proposition} \label{prop:integral} 
 If $c\in Z_2(\G,\Z),$ then $\langle[\rho],[c]\rangle\in H$.
 \end{proposition}
\begin{proof} By Proposition~\ref{prop:Hopf}, we may take $c=\varphi(r)$ with $r$ a commutator and we have
\begin{equation}\label{eq:vanish-q}
 q(\langle[\rho],[\varphi(r)]\rangle)=q\circ L\left(\prod_{i=1}^{g} [\rho(\bar{a}_i),\rho(\bar{b}_i)]\right)=\Delta\left(\prod_{i=1}^{g} [\rho(\bar{a}_i),\rho(\bar{b}_i)]\right)=0  
 \end{equation}
 since each $\Delta_{\tau_n}$ vanishes on the commutator subgroup $[\GL_{\infty}(A_n),\GL_{\infty}(A_n)]$.
\end{proof}
\begin{corollary}
   $[\rho]$ belongs to the image of the map $H^2(\G,H)\to H^2(\G,Q(\C))$.
\end{corollary} 
\begin{proof}
   Since the sequence 
  \[H^2(\G,H)\to H^2(\G,Q(\C))\stackrel{q_*}\longrightarrow 
 H^2 (\G, Q(\C)/H)),\]
  is exact in the middle and both  $Q(\C)$ and $Q(\C)/H$ are divisible groups, using the universal coefficient theorem,
  it suffices to show that the image of $[\rho]$ in $H^2 (\G, Q(\C)/H))$ pairs trivially with
  any 2-cycle $c\in Z_2(\G,\Z)$. But this is exactly what  we have verified in equation~\eqref{eq:vanish-q}.
  \end{proof}
\begin{remark}
    As a consequence of Proposition ~\ref{prop:integral}, we see that if $c$ is a 2-cycle as in \eqref{eq:pre}, then for all sufficiently large $n:$
    $$\sum_{j=1}^m\frac{1}{2 \pi i}  k_j\tau_{n}\left(\log \left(\rho_n(a_j)\rho_n(b_j)\rho_n(a_jb_j)^{-1}\right)\right)=\tau_{n*}(y_n)\quad \text{for some}\quad y_n\in K_0(A_n).$$ It is the natural to inquire how  $[c]\in H_2(G,\Z)$ is related   $y\in K_0(A_n)$. Assuming that $A_n$ are $C^*$-algebras, we shall explain this by invoking an index formula from \cite{DDD}, which we review below.
\end{remark}
Let
$\beta^{\G}:  H_2(\G,\Z) \cong H_2(B\G,\Z) \to RK_0(B\G)$ be  the (rationally injective) homomorphism
 studied in \cite{Bettaieb-Matthey-Valette},   \cite{MR1951251} and let $\alpha^\G : H_2(\G,\Z) \to K_0(\ell^1(\G))$ be the composition \mbox{$\alpha^\G=\mu_1^\G  \circ \beta^\G$} where $\mu_1^\G:RK_0(B\G)\to K_0(\ell^1(\G))$ is  the $\ell^1$-version of the assembly map of Lafforgue \cite{Lafforgue}.
The  linear extension $\rho:\ell^1(\G)\to M_n(\C)$ of a sufficiently multiplicative unital map
 $\rho:\G \to U(n)$ satisfies the following: 
 \vskip 4pt
\begin{theorem}[\cite{DDD}]\label{thm:ddd} 
\emph{ Let $x\in H_2(\G,\Z)$ be represented by  a product of commutators
 $\prod_{i=1}^{g} [a_i,b_i]$ with $a_i,b_i \in F$ and $\prod_{i=1}^{g} [\bar{a}_i,\bar{b}_i]=1$.
 There exist a finite set $S \subset \G$  and $\varepsilon>0$ such that if $(A,\tau)$ is a tracial $C^*$-algebra and if $\rho:\G \to U(A)$ is a unital map  with
 $\|\rho(st)-\rho(s)\rho(t)\|<\varepsilon$ for all $s,t \in S$, then}
 
  \begin{equation}\label{eq:indexc0}
  \tau_*(\rho_\sharp( \alpha^\G(x)))=\frac{1}{2\pi i}\tau\left( \log \left(\prod_{i=1}^{g} [\rho(\bar{a}_i),\rho(\bar{b}_i)]\right)\right).
 \end{equation}
 \end{theorem}
 Here, if we write $\alpha^\G(x)=[p_0]-[p_1],$ where $p_i$ are projections in matrices over $\ell^1(\G)$, then  $\rho_\sharp( \alpha^\G(x)) = \rho_\sharp(p_0) -
 \rho_\sharp(p_1),$ where $\rho_\sharp(p_i)\in K_0(A)$ is the K-theory class of the perturbation of $(\mathrm{id}\otimes \rho )(p_i)$ to a projection via analytic functional calculus.
 
From Theorem~\ref{thm:ddd} we see that if
  $c=\sum_{j=1}^{m}k_j [a_j|b_j]\in Z_2(\G,\Z)$, and if $x=\varphi^{-1}(c)$ is the corresponding element in the Hopf description of $H^2(\G,\Z)$,
  then for all sufficiently large $n$,
  \begin{equation*}\label{eq:indexc01}
 \tau_{n*}(\rho_{n\sharp}( \alpha^\G(x)) =\sum_{j=1}^m\frac{1}{2 \pi i}  k_j\tau_{n}\left(\log \left(\rho_n(a_j)\rho_n(b_j)\rho_n(a_jb_j)^{-1}\right)\right).
 \end{equation*}
   
Let $A=\ell^\infty(\N,A_n)$, $J=c_0(\mathbb{N},A_n)$
and $\tau:A\to \ell^\infty(\N,\C)$ be as in 
 Section~\ref{sec:3}.
 
Then we have an induced maps 
\[\dot\rho_{*}: K_0(\ell^1(\G))\to K_0(A/J)\quad \text{and} \quad \tau_*:K_0(A/J)\to \ell^\infty(\N,\C)/c_{00}(\N,C).\]

\begin{corollary}\label{index}
For any $c\in Z_2(\G,\Z)$, let $x=\varphi^{-1}(c)$ be the corresponding element in the Hopf description of $H^2(\G,\Z)$.
  Then \begin{equation}\label{eq:indexc01+}
     \tau_{*}(\dot\rho_{*}( \alpha^\G(x)) =\langle [\rho],[c]\rangle.
 \end{equation}
\end{corollary}

\section{C*-Algebra non-stability for groups with nontrivial 2-cohomology}\label{sec:5}
\begin{definition}\label{def:Cstar-stable}
	Let  $\G$ be a countable discrete group and let $\mathcal{B}$ be a class of unital $C^*$-algebras.
	\begin{itemize}
		\item[(a)]	 $\G$ is called local-to-local $\mathcal{B}$-stable if  for any sequence of unital maps
		$\{\rho_n:\G \to U(B_n)\}$ with $B_n\in\mathcal{B}$  and
		\begin{equation}\label{a1}\lim_{n\to \infty} \|\rho_n(st)-\rho_n(s)\rho_n(t)\|=0, \quad \text{for all}\,\, s,t\in \G,\end{equation}
		there exists a sequence of homomorphisms $\{\pi_n:\G\to U(B_n)\}$ such that
		\begin{equation}\label{a2}\lim_{n\to \infty}  \|\rho_n(s)-\pi_n(s)\|=0, \quad \text{for all}\,\, s\in \G.\end{equation}
		\item[(b)] $\G$ is called uniform-to-local $\mathcal{B}$-stable if for any sequence of unital maps
		$\{\rho_n:\G \to U(B_n)\}$ that satisfies 
		\begin{equation}\label{a11}\lim_{n\to \infty} \left(\sup_{s,t\in \G}  \|\rho_n(st)-\rho_n(s)\rho_n(t)\|\right)=0, \end{equation}
		there exists a sequence of homomorphisms $\{\pi_n:\G\to U(B_n)\}$  satisfying \eqref{a2}.
	\end{itemize}
\end{definition}
\begin{remark}\label{rem:stability}
	Local-to-local $\mathcal{B}$-stability  is equivalent to asking that for any finite set $F\subset \G$ and any $\ep>0,$ there exist a finite set $S\subset \G$ and $\delta>0$ with the property that for any  unital map $\rho:\G \to U(B),$ $B\in \mathcal{B},$ with $\|\rho(s)\rho(t)-\rho(s)\|<\delta$ for all $s,t \in S$, there is a representation $\pi:\G \to U(B)$ such that $\|\rho(s)-\pi(s)\|<\ep$ for all $s\in F.$ 
	Uniform-to-local stability admits a characterization in the same vein. A map  $\rho$ as above is called an $(S,\delta)$-representation.
\end{remark}
Let $Q$ be an abelian group with trivial $\G$-action.
If $Q$ is a subgroup of $\R$ (such as $\Z$, $\Q$ or $\R$), we consider the bounded cohomology group $H_b^2(\G,Q),$ which defined similarly to  $H^2(\G,Q),$
except that all cocycles and coboundaries are required to be bounded functions $\G^k \to Q \subset \R.$ There is an obvious canonical comparison map $J:H_b^2(\G,Q)\to H^2(\G,Q).$
%%%%%%%%
\begin{notation}\label{notation}  
	(Projective representations associated to  2-cocycles.)
	Let $\sigma\in Z^2(\G,\R)$ be  a normalized $2$-cocycle and let \(0 \to \R \to E \stackrel{q}{\longrightarrow} \Gamma \to 1\) be its associated extension; see preliminaries for a construction.  
	For each $\theta\in [0,1],$ we have a character $\R \to \mathbb{T},$ $r\mapsto e^{2\pi i\theta r}$ and the corresponding $2$-cocycle $\sigma_\theta\in Z^2(\G,\mathbb{T}),$ $
	\sigma_\theta(s,t)=e^{2\pi i\theta\sigma(s,t)}.$ 
	Consider the 
	 twisted full group $C^*$-algebra, $C^*(\Gamma,\sigma_{\theta})$
	and its quotient map onto the reduced twisted group $C^*$-algebra, $p:C^*(\Gamma,\sigma_{\theta})\to C_r^*(\Gamma,\sigma_{\theta})$, \cite{Packer-Raeburn}.
	By construction, there is a unital map $\rho_\theta:\G \to C^*(\Gamma,\sigma_{\theta})$
	that satisfies the equation 
	\begin{equation}\label{eq:projective}
		\rho_\theta (s)\rho_\theta (t)\rho_\theta (st)^{-1}=e^{2 \pi i\theta \sigma(s,t)},\,\, s,t \in \G.
	\end{equation}
	The map $\bar\rho_\theta:=p\circ \rho_\theta:\G \to C^*_r(\G,\sigma_\theta)$
	coincides with
	the composition of the induced representation $\mathrm{Ind}_{Q}^E (\omega_{\theta})$
	with the section $\gamma$ and in fact $\bar\rho_\theta:\G \to U(C_r^*(\G,\sigma_\theta))\subset U(\ell^2(\G)),$ is given by $\bar\rho_\theta (s)\delta_t=\sigma_\theta (s,t)\delta_{st}.$
	There is a canonical tracial state $\tau:C_r^*(\G,\sigma_{\theta})\to \C$ induced by vector state $\tau(a)=\langle a \delta_e,\delta_e\rangle.$
    This induces a tracial state on $C^*(\G,\sigma_{\theta}).$
    The corresponding twisted von Neumann algebra  denoted $L(G,\sigma_{\theta}),$ is the weak closure of $C_r^*(\G,\sigma_{\theta}).$
\end{notation}
Choose a sequence
	$\theta_n \searrow 0$ together with  a sequence of unital homomorphisms $\pi_n: C^*(\G,\sigma_{\theta_n})\to A_n$ to tracial $C^*$-algebras $(A_n,\tau_n)$ with $\tau_n(1_n)=1.$
For any choice of $(\theta_n),$ there exist such $C^*$-algebras $A_n$ and corresponding maps $(\pi_n)$.
	Indeed, one may choose $A_n=C_r^*(\G,\sigma_{\theta_n})$ or $A_n=L(\G,\sigma_{\theta_n})$. Subsequently, we will discuss examples where $A_n$ are matricial $C^*$-algebras $M_{k(n)}$. 
The sequence $\rho_n:=\pi_n \circ \rho_{\theta_n}: \G \to U(A_n)$  satisfies  the conditions 
 \begin{equation}\label{eqn:gestimate}
		\rho_n(s)\rho_n (t)\rho_n (st)^{-1}=e^{2 \pi i\theta_n \sigma(s,t)} 1_n\,\, \text{and} \quad \lim_{n\to \infty}\|\rho_{n}(s)\rho_{n}(t)-\rho_{n}(st)\|= 0, \quad\forall \, s,t \in \G.
	\end{equation}
\begin{proposition}\label{prompt}
  If $\rho=(\rho_n)_n$ is a sequence of projective representations as in \ref{notation}, 
 then  $[\rho]$ is the image of $[\sigma]$ under the injective map $q_*:H^2(\G,\R)\to H^2(\G,Q(\R))$ induced by the homomorphism $q:\R\to Q(\R)$, $r\mapsto (\theta_nr)_n$.
\end{proposition}
\begin{proof}
 By Corollary~\ref{cor:proj-appl}, $[\rho]\in H^2(\G,Q(\R))$ is represented by the cocycle
 $(\theta_n \sigma)_n$.
 Let us verify that $q_*$ is injective. Suppose $[\sigma]\neq 0$.
By the universal coefficient theorem, $H^2(\G,\R)\cong \Hom(H_2(\G,\Z),\R),$ and hence there is a 2-cocycle $c\in Z_2(\G,\Z)$  with nontrivial Kronecker pairing  $\langle\sigma,c\rangle\neq 0.$
Then $ \langle q_*(\sigma),c\rangle=(\theta_n \langle\sigma,c\rangle)_n \neq 0.$
\end{proof}
The following is Theorem~\ref{thm:Ai} from the introduction. We prove it simultaneously with Theorem~\ref{thm-vN}.
\begin{theorem}\label{thm:A} Let  $\G$ be a discrete countable group.
	\begin{itemize}
		\item[(1)] If $ H^2(\Gamma,\R)\neq 0$, then $\Gamma$ is not local-to-local stable with respect the class of separable unital tracial $C^*$-algebras.
		\item[(2)] If the comparison map $J: H^2_b(\Gamma,\R)\to H^2(\Gamma,\R)$ is nonzero, then $\Gamma$ is not uniform-to-local stable with respect the class of separable unital tracial $C^*$-algebras.
	\end{itemize} 
\end{theorem}
\begin{proof} 
Throughout the proof we take either $A_n=C_r^*(\G,\sigma_{\theta_n})$ or $A_n=L(\G,\sigma_{\theta_n})$.
The later case yields the proof for Theorem~\ref{thm-vN}.
	(1) 
Let $\sigma\in Z^2(\G,\R)$ be  a normalized $2$-cocycle with $[\sigma]\neq 0$.
Let $\rho=(\rho_n)_n$ be a sequence of projective representations associated to $\sigma$ as in \ref{notation}. For example we may choose $\theta_n=1/n$.
Then $[\rho]=q_*[\sigma]\neq 0$ by Proposition~\ref{prompt}.
By Theorem~\ref{thm:basic}
 there exists no sequence of group homomorphisms $\pi_n: \G \to \GL(A_n)$ such that $\|\pi_n(s)-\rho_{n}(s)\|\to 0$, for all $ s \in \G.$

	(2) We are using notation and arguments from the proof of (1) above. If the comparison map is nonzero, then there is 2-cocycle $\sigma$ such that $\sup_{s,t\in \G}|\sigma(s,t)|<\infty$ and $[\sigma]\neq 0$ in $H^2(\G,\R).$
  In this case equation $\rho_n(s)\rho_n (t)\rho_n (st)^{-1}=e^{2 \pi i\theta_n \sigma(s,t)}1_n$ implies that
	\[\lim_{n\to \infty} \sup_{s,t\in \G}\|\rho_{n}(s)\rho_{n}(t)-\rho_{n}(st)\|= 0.\]
	The same argument as in the proof of (1) shows that  there exists no sequence of group homomorphisms $\pi_n: \G \to \GL(A_n)$ such that $\|\pi_n(s)-\rho_{n}(s)\|\to 0$, for all $ s \in \G.$
\end{proof}

\section{Matrix non-stability}\label{sec:6}
Matrix stability refers to stability with respect to the class
$\mathcal{B}=\{M_n(\C):n\geq 1\}.$
In section we discuss two applications to matrix stability.

Following \cite{arXiv:2204.10354}, 
if $\Lambda$ is a countable discrete group we say $x\in H^2(\Lambda;\Z)$ of finite type if it is given by a central extension
\begin{equation}
	\label{eq:ft}
	\begin{tikzcd}
		0\arrow{r} & \Z
		\arrow{r}{\iota} &
		{E} \arrow{r}{q}
		& \Lambda\arrow{r}{}& 1.
	\end{tikzcd}
\end{equation}
such that $E$ has a sequence of finite index subgroups $\{E_k\}_{k\in\N}$ so that $\iota(\Z)\cap\bigcap_{k}E_k=\{e\}.$
If $E$ is residually finite, then $x$  is of finite type.
It was shown in \cite{arXiv:2204.10354} that for a finitely generated nilpotent group $\Lambda$, all the elements of  
$H^2(\Lambda;\Z)$ are of finite type.

\begin{lemma}\label{lemma:pullback}
	If $\varphi:\G\rightarrow\Lambda$ is a group homomorphism and $y\in H^2(\Lambda,\Z)$ is of finite type, then so is $\varphi^*(y)\in H^2(\G,\Z)$. 
\end{lemma}
\begin{proof}
	Representing $x$ as a central extension as in \eqref{eq:ft}, then, as explained in  \cite[chapter IV.3 Exercise 1]{iBrown:book-cohomology}, $\varphi^*(x)$ corresponds to a central extension as in the diagram
	\[
	\begin{tikzcd}
		0 \ar[r]&	\Z\ar[r, "\widetilde{\iota}"] \ar[equal]{d}& \widetilde{E}
		\ar[r, "\widetilde{q}"]\ar[d,"\widetilde{\varphi}"] &  \G\ar[r]\ar[d, "{\varphi}"] & 0\\
		0 \ar[r]&	\Z\ar[r, "{\iota}"] & E\ar[r,"{q}"] &  \Lambda\ar[r] & 0
	\end{tikzcd}
	\]
	If  $(E_k)$  is a family of finite index subgroups of ${E}$ so that ${\iota}(\Z)\cap \left(\bigcap_k {E}_k \right)=\{e\}$, then $(\widetilde{\varphi}^{-1}(E_k))$ is a family of finite index subgroups of $\widetilde{E}$ and $\widetilde{\iota}(\Z)\cap \left(\bigcap_k \widetilde{\varphi}^{-1}({E}_k) \right)=\{e\}.$
 \end{proof}

For $1\le p<\infty,$ the unnormalized Schatten $p$-norm on $M_n(\C)$ is $||M||_p=\Tr(|M|^p)^{1/p}$. We can extend this to $p=\infty$ by defining the Schatten $\infty$-norm to be the operator norm. Stability in unnormalized Schatten $p$-norm is defined in the same way one defines matricial stability but with the operator norm replaced by the $p$-norm. Of particular interest is the unnormalized 2-norm, also called the Frobenius norm. 
We analogously define uniform-to-local stability in the unnormalized Schatten \( p \)-norm, similar to uniform-to-local matricial stability, but with the operator norm replaced by the unnormalized Schatten \( p \)-norm.

\begin{corollary}\label{cor:321} \, \begin{enumerate}
		\item Let $\G$ be a finitely generated group such that   $ H^2(\G;\Z)$ has a finite type element $x$ of infinite order. Then $\G$ is not local-to-local (matricially) stable. If we furthermore assume that
 $\G$ is  hyperbolic,
		then $\G$ is not uniform-to-local (matricially) stable. 	
		\item Suppose that $\G$ is a  hyperbolic group, with $\beta,\psi\in H^1(\G,\Z)$ so that $\beta\smile\psi\in H^2(\G,\Z)$ is non-torsion. Let $1<p\le\infty$. Then $\G$ is not uniform-to-local stable in the unnormalized Schatten $p$-norm.
	\end{enumerate}
\end{corollary}
\begin{proof} {(1)} We prove only the second part of (1). The first part is proved by a similar argument in~\cite{arXiv:2204.10354} and in fact it is generalized by Theorem~\ref{tprop:A}.
	Represent $x$ by a central extension as in \eqref{eq:ft}.
	Since the comparison map $J: H^2_b(\Gamma,\R)\to H^2(\Gamma,\R)$
	is surjective, so is the map $J: H^2_b(\Gamma,\Z)\to H^2(\Gamma,\Z)$, \cite[ch.2]{Frigerio}. Consequently, we may arrange that this extension is constructed from a bounded $2$-cocycle $\sigma,$ $|\sigma(s,t)|\leq M$ for all $s,t \in \G$.
	In particular, we have a section $\gamma:\G \to {E}$ of ${q}$ such that
 \(\gamma(s)\gamma(t)\gamma(st)^{-1}={\iota}(\sigma(s,t)).\)
	Since the extension that represents $x$ is of finite type, as shown in \cite{arXiv:2204.10354}, there is a sequence of finite dimensional representations  $\phi_n:E \to U(k_n),$ $k_n\to \infty$ such that \[\phi_n(\iota(1))=e^{2\pi i\theta_n}\,\, \text{for some }\,\, \theta_n \searrow 0.\]
	Consider the sequence of unital maps $\rho_n:=\Phi_n\circ \gamma:\G \to U(k_n).$ They satisfy
	\begin{equation*}\rho_n (s)\rho_n(t)\rho_n (st)^{-1}=e^{2 \pi i\theta_n \sigma(s,t)}\quad \text{and hence}\,\,\lim_{n\to \infty} \left(\sup_{s,t\in \G}  \|\rho_n(st)-\rho_n(s)\rho_n(t)\|\right)=0. \end{equation*}
 Since \(\langle [\rho],[c]\rangle=(\theta_n\tau_n(1_n)\langle\sigma,c\rangle)_n\neq 0,\)   by the same calculation as in the proof of Proposition~\ref{prompt}, one cannot perturb $(\rho_n)$ to a sequence of representations.

	{(2)} Now suppose that $x\in H^2(\G;\Z)$ can be written as $\alpha\smile\beta$. Here since $H^1(\G,\Z)\cong\Hom(\G,\Z)$ we can view $\alpha$ and $\beta$ as homomorphisms from $\G$ to $\Z$. Define $\varphi:\G \to \Z^2$ by   $\varphi(s)=(\alpha(s),\beta(s)).$  Note note that $H^1(\Z^2,\Z)\cong\Hom(\Z^2,\Z)$ is generated by the two canonical projections from $\Z^2$ to $\Z$ denoted $\pi_1$ and $\pi_2$, and $H^2(\Z^2,\Z)$ is generated by $\pi_1\smile\pi_2$. It follows that $\varphi^*(\pi_1\smile\pi_2)=\varphi^*(\pi_1)\smile\varphi^*(\pi_2)=\alpha\smile\beta=x$.  
	
	By \cite[chapter IV.3 Exercise 1]{iBrown:book-cohomology} there is a commutative diagram
	\begin{equation*}\begin{tikzcd}
			0\arrow{r} & \Z
			\arrow{r}{\iota}\ar[equal]{d} &
			E\arrow{d}{\Tilde\varphi} \arrow{r}{\tilde q}
			& \Gamma\arrow{r}{}\arrow{d}{\varphi}& 0
			\\
			0\arrow{r} & \Z
			\arrow{r}{} &
			\HH_3 \arrow{r}{q}
			& \Z^2\arrow{r}{}&0
		\end{tikzcd}\end{equation*}
	
	Here $\HH_3$ is the discrete Heisenberg group, generated by $a,b,z$ with the relations $ba=abz$, $za=az$, and $zb=bz$. 
	
	Pick a set-theoretic section $\gamma:\G\rightarrow E$ corresponding to a normalized bounded 2-cocycle $\sigma$. Set $\lambda_n=e^{2\pi i/n}$ and let $\psi_n$ the representation of $\HH_3$ (\cite{Kazhdan-epsilon},  \cite{Voi:unitaries}) defined by 
	$$
	\psi_n(a)=\left(\begin{matrix}
		0 & 0 &  0  &\cdots &{1} \cr
		{1}& 0 &  0  &\cdots & 0 \cr
		0 &{1} &  0  &\cdots & 0 \cr
		%0 &0 &  {1} &\cdots& 0 \cr
		&  &  &\cdots & \cr
		0 &  0&  \cdots  &{1}& 0 \cr
	\end{matrix}\right),\,\,
	\psi_n(b)=\left(\begin{matrix}
		{\lambda_n} & 0 &  0  & 0 & 0\cr
		0 & {\lambda_n^2} &  0  & 0 & 0\cr
		%0 &0& {\lambda_n^3} & 0 & 0\cr
		0 &  0&  \lambda_n^3 &{\cdot}& 0\cr
		&  &  &\cdots & \cr
		0 &  0&  0 & {\cdots}& {\lambda_n^n}\cr
	\end{matrix}\right),\quad \psi_n(z)=\lambda_n1_n.
	$$
	 Then for $\rho_n=\psi_n\circ\tilde\varphi\circ\gamma$ we have
 $\rho_n(s)\rho_n(t)\rho_n(st)^{-1}=e^{2\pi i\sigma(s,t)/n}.$ It follows that
 \[||\rho_n(s)\rho_n(t)-\rho_n(st)||_p\le||\rho_n(s)\rho_n(t)-\rho_n(st)||n^{1/p}\le\frac{2\pi|\sigma(s,t)|}{n}n^{1/p}\]and this converges to zero uniformly in $s$ and $t$. The rest of the proof is identical to the proof of {(1)}.
\end{proof}
\begin{remark}
	Corollary~\ref{cor:321} reproves a result of Kazhdan \cite{Kazhdan-epsilon}. Indeed, if $\G_g$ is a surface group of genus $g\geq 2$, 
	then the generator $x$ of $H^2(\G_g,\Z)$ is of finite type. This is seen by applying Lemma~\ref{lemma:pullback} with 
	$\Lambda=\Z^2$ and $y$ the generator of $H^2(\Z^2,\Z)$ corresponding to the central extension 
	\begin{equation*}
		\begin{tikzcd}
			0\arrow{r} & \Z
			\arrow{r} &
			{\mathbb{H}_3} \arrow{r}
			& \Z^2\arrow{r}{}& 0.
		\end{tikzcd}
	\end{equation*}
	It is of finite type since the Heisenberg group $\mathbb{H}_3$ is residually finite.
\end{remark}
The following is Theorem~\ref{tprop:Ai} from the introduction.
\begin{theorem}\label{tprop:A}
	If $\G$ is a countable discrete group and for some $[\sigma]\in H^2(\Gamma,\R)\setminus \{0\}$  there  is a sequence 
	$\theta_n \searrow 0$  so that for each $n$,  the twisted full group $C^*$-algebra, $C^*(\Gamma,\sigma_{\theta_n})$ has a nonzero MF quotient, then $\Gamma$ is not matricially stable.
\end{theorem}
\begin{proof} Seeking a contradiction, suppose that $\G$ is matricially stable. Write $\G$ as the union of an increasing sequence of finite sets $F_n$ and let $\ep_n=1/n.$ By matricial stability, there is an increasing sequence of finite sets $(S_n)$ and a sequence $0<\delta_n \searrow 0$ such that for any $(S_n,\delta_n)$-representation $\rho_n:\G \to U(k_n)$ there is
	a true representation $\pi_n:\G\to U(k_n)$ such that $\|\rho_n(s)-\pi_n(s)\|<\ep_n$ for all $s\in F_n$.
	
	By assumption, for each $n$ there is a unital homomorphism $C^*(\Gamma,\sigma_{\theta_n})\to A_n$ where $A_n$ is a $C^*$-algebra of the form 
	$$ A_n =\frac{\prod  M_{k_n(i)}}{\oplus M_{k_n(i)}}$$
	for which we fix a tracial state. Consider the family $\mathcal{B}:=\{A_n:n\geq 1\}$.
	As shown in the first part of the proof of Theorem~\ref{thm:A},   
	since $\theta_n \searrow 0,$ we can find a sequence $(D_n)$ of $C^*$-algebras in $\mathcal{B}$ together with a sequence of unital $(S_n,\frac{1}{2}\delta_n)$-representations $\rho_n:\G \to U(D_n)$ for which there is no sequence of representations $\pi_n:\G \to U(D_n)$ such that $\|\rho_n(s)-\pi_n(s)\|\to 0$ for all $s\in \G$. Write $D_n=P_n/Q_n$
	where \[P_n=\prod  M_{k_n(i)}\quad \text{and} \quad  Q_n=\bigoplus M_{k_n(i)}.\]
	Let $q_n:P_n \to D_n$ be the quotient map.
	It is routine to find a sequence of unital $(S_n,\delta_n)$-representations $\rho'_n:\G \to U(P_n)$ such that $q_n\circ \rho_n'=\rho_n.$ By our choice of $(S_n,\delta_n),$ it follows that there is a sequence of representations $\pi_n':\G \to U(P_n)$ such that \(\|\rho_n'(s)-\pi_n'(s)\|<\ep_n\) for all $s\in F_n$.
	Thus if we set $\pi_n:=q_n \circ \pi_n', $
	then $(\pi_n)$ is a sequence of representations $\pi_n:\G \to U(D_n)$ such that $\|\rho_n(s)-\pi_n(s)\|\to 0$ for all $s\in \G$. This contradicts our choice of $(\rho_n)$. 
\end{proof}
{For a more direct (but somewhat more technical)  alternate proof   of Theorem~\ref{tprop:A} one can proceed as follows.  Consider the sequence $(\rho_n)$ defined in \eqref{eqn:gestimate} where $A_n$ are MF-algebras.
 There is a 2-cocycle $c\in Z_2(\G,\Z)$  with nontrivial Kronecker pairing  $\langle\sigma,c\rangle\neq 0$.
	As seen earlier, the pairing 
	$H^2(\G,Q(\C))\times H_2(\G,\Z)\to Q(\C)$ yields
    \[\langle [\rho],[c]\rangle=(\theta_n\tau_n(1_n)\langle\sigma,c\rangle)_n\neq 0.\]	
Since $A_n$ is MF, for each $n\geq 1$ there is a sequence of unital maps $\{\rho_{n,i}:\G \to U(m_{n,i})\}_i$ such that
for all $n\geq 1$,
\[\lim_i \|\rho_{n,i}(s)\rho_{n,i}(t)\rho_{n,i}(st)^{-1}-e^{2\pi i \theta_n\sigma(s,t)}1\|=0,\quad \text{for all}\quad s,t\in \G.\]
We shall use the fact  that for any unital Banach algebra $A$,  if   $u,v\in A$ with $\|u-1\|<1/2$ and $\|v-1\|<1/2$, then $\|\log u-\log v\|<2\|u-v\|$.
Since the support of $c$ as in \eqref{eq:pre} is finite, we can find a fast increasing sequence $(i_n)$ such that for each $n,$
if we define $\rho_n'=\rho_{n,i_n}$, 
 and $\lambda_n$ is the $n^{th}$ component of $\langle [\rho'],[c]\rangle$, namely
 \begin{equation*}
\lambda_n=\sum_{j=1}^mk_j\frac{1}{2 \pi i}  \tau_{n}\left(\log \left(\rho_{n,i_n}(a_j)\rho_{n,i_n}(b_j)\rho_{n,i_n}(a_jb_j)^{-1}\right)\right),
\end{equation*}
 then 
  $|\lambda_n-\theta_n\tau_n(1_n)\langle\sigma,c\rangle|<\frac{1}{2}\theta_n\tau_n(1_n)|\langle\sigma,c\rangle|$
  and hence $\lambda_n\neq 0$.
We conclude the proof by applying Theorem~\ref{thm:basic} to $\rho'$.}

 %\newpage
\addcontentsline{toc}{section}{References} 
{\small
\bibliographystyle{abbrv}
%\bibliography{Operator24}

\begin{thebibliography}{10}

\bibitem{Bettaieb-Matthey-Valette}
H.~Bettaieb, M.~Matthey, and A.~Valette.
\newblock Unbounded symmetric operators in {$K$}-homology and the
  {B}aum-{C}onnes conjecture.
\newblock {\em J. Funct. Anal.}, 229(1):184--237, 2005.

\bibitem{iBrown:book-cohomology}
K.~S. Brown.
\newblock {\em Cohomology of groups}, volume~87 of {\em Graduate Texts in
  Mathematics}.
\newblock Springer-Verlag, New York, 1994.
\newblock Corrected reprint of the 1982 original.

\bibitem{BOT}
M.~Burger, N.~Ozawa, and A.~Thom.
\newblock On {U}lam stability.
\newblock {\em Israel J. Math.}, 193(1):109--129, 2013.

\bibitem{CCC}
M.~Dadarlat.
\newblock Obstructions to matricial stability of discrete groups and almost
  flat {K}-theory.
\newblock {\em Adv. Math.}, 384:Paper No. 107722, 29, 2021.

\bibitem{DDD}
M.~Dadarlat.
\newblock Quasi-representations of groups and two-homology.
\newblock {\em Comm. Math. Phys.}, 393(1):267--277, 2022.

\bibitem{EEE}
M.~Dadarlat.
\newblock Non-stable groups.
\newblock to appear in {\em M{\"u}nster J. Math.}, arXiv:2304.04645 [math.OA],
  2023.

\bibitem{DECHIFFRE}
M.~De~Chiffre, L.~Glebsky, A.~Lubotzky, and A.~Thom.
\newblock S{tability}, {cohomology} {vanishing}, {and} {nonapproximable}
  {groups}.
\newblock {\em Forum Math. Sigma}, 8:e18, 2020.

\bibitem{paper:DeLaHarpeSkandalis}
P.~de~la Harpe and G.~Skandalis.
\newblock Produits finis de commutateurs dans les {$C^\ast$}-alg\`ebres.
\newblock {\em Ann. Inst. Fourier (Grenoble)}, 34(4):169--202, 1984.

\bibitem{ESS-published}
S.~Eilers, T.~Shulman, and A.~P.~W. S{\o}rensen.
\newblock {$C^*$}-stability of discrete groups.
\newblock {\em Adv. Math.}, 373:107324, 41, 2020.

\bibitem{Frigerio}
R.~Frigerio.
\newblock {\em Bounded cohomology of discrete groups}, volume 227 of {\em
  Mathematical Surveys and Monographs}.
\newblock American Mathematical Society, Providence, RI, 2017.

\bibitem{MR4600056}
F.~Glebe.
\newblock Frobenius non-stability of nilpotent groups.
\newblock {\em Adv. Math.}, 428:Paper No. 109129, 21, 2023.

\bibitem{arXiv:2204.10354}
F.~Glebe.
\newblock A constructive proof that many groups with non-torsion 2-cohomology
  are not matricially stable.
\newblock {\em Groups, Geometry, and Dynamics}, 2024.

\bibitem{Ulam-new}
L.~Glebsky, A.~Lubotzky, N.~Monod, and B.~Rangarajan.
\newblock {Asymptotic Cohomology and Uniform Stability for Lattices in
  Semisimple Groups}.
\newblock arXiv:2301.00476 [math.GR], 2022.


\bibitem{Ioana-S-W:coho}
A.~Ioana, P.~Spaas, and M.~Wiersma.
\newblock {Cohomological obstructions to lifting properties for full group
  C$^*$-algebras}.
\newblock {https://arxiv.org/abs/2006.01874} [math.OA], 2020.

\bibitem{Kazhdan-epsilon}
D.~Kazhdan.
\newblock On {$\varepsilon $}-representations.
\newblock {\em Israel J. Math.}, 43(4):315--323, 1982.

\bibitem{Lafforgue}
V.~Lafforgue.
\newblock {$K$}-th\'eorie bivariante pour les alg\`ebres de {B}anach et
  conjecture de {B}aum-{C}onnes.
\newblock {\em Invent. Math.}, 149(1):1--95, 2002.

\bibitem{MR1951251}
M.~Matthey.
\newblock Mapping the homology of a group to the {$K$}-theory of its
  {$C^*$}-algebra.
\newblock {\em Illinois J. Math.}, 46(3):953--977, 2002.
\bibitem{Mineyev}
I.~Mineyev.
\newblock {Straightening and bounded cohomology of hyperbolic groups.}
\newblock {\em Geom. Funct. Anal.}, 11(4):807--839, 2001.


\bibitem{Packer-Raeburn}
J.~A. Packer and I.~Raeburn.
\newblock On the structure of twisted group {$C^*$}-algebras.
\newblock {\em Trans. Amer. Math. Soc.}, 334(2):685--718, 1992.

\bibitem{Voi:unitaries}
D.~Voiculescu.
\newblock {Asymptotically commuting finite rank unitary operators without
  commuting approximats}.
\newblock {\em Acta Sci. Math. (Szeged)}, 45:429--431, 1983.

\end{thebibliography}

}
\end{document}